\documentclass[reqno,11pt]{amsart}
\usepackage[colorlinks=true,urlcolor=blue,
citecolor=red,linkcolor=blue,linktocpage,pdfpagelabels,
bookmarksnumbered,bookmarksopen]{hyperref}
\usepackage{enumerate}
\usepackage[left=2.6cm,right=2.6cm,top=2.9cm,bottom=2.9cm]{geometry}

\usepackage{amsmath,amssymb,latexsym,soul,cite,mathrsfs}
\numberwithin{equation}{section}  
\usepackage{color,enumitem,graphicx}
\usepackage{enumerate}
\usepackage{cancel}

\usepackage[colorlinks=true,urlcolor=blue,
citecolor=red,linkcolor=blue,linktocpage,pdfpagelabels,
bookmarksnumbered,bookmarksopen]{hyperref}
\usepackage[english]{babel}

\newtheorem{theorem}{Theorem}[section]
\newtheorem{proposition}[theorem]{Proposition}
\newtheorem{corollary}[theorem]{Corollary}
\newtheorem{lemma}[theorem]{Lemma}
\DeclareMathAlphabet{\mathpzc}{OT1}{pzc}{m}{it}

\newtheorem{example}[theorem]{Example}

\newtheorem{remark}[theorem]{Remark}

\newcommand{\tr}{\text{tr}}
\renewcommand{\div}{\text{div}}

\title[Integral Inequalities and Rigidity for $V$-Static-Type Equations on Manifolds with Boundary
]{Integral Inequalities and Rigidity for $V$-Static-Type Equations on Manifolds with Boundary}
\author[M. Andrade]{Maria Andrade}
\address[M. Andrade]{Departamento de Matemática, 
 Universidade Federal de Sergipe, 49100-000, S\~ao Cristov\~ao-SE, Brazil.
}
\email{\href{mailto: maria@mat.ufs.br}{ maria@mat.ufs.br}}

\address{Current address: Department of Mathematics, Princeton University, Princeton, NJ, USA, 08544.}
\email{\href{mailto: ma6208@princeton.edu}{ ma6208@princeton.edu}}


\begin{document}
\maketitle
\begin{abstract}
In this work, we study compact Riemannian manifolds with boundary satisfying V-static-type equations. By combining a generalized Reilly formula with Steklov-type boundary value problems, we derive integral inequalities for geometric quantities associated with the boundary. These inequalities lead to rigidity results, including characterizations of geodesic balls in space forms. In particular, our results offer new insights into several known rigidity theorems in the literature.

\end{abstract}


\section{Introduction}

A classical result due to Alexandrov asserts that a closed embedded hypersurface with constant mean curvature in Euclidean space must be a round sphere. Alexandrov's proof relies on the maximum principle for elliptic equations. Subsequent developments showed that rigidity phenomena can also be derived from integral identities. In a celebrated paper \cite{reilly1977applications}, Reilly established an integral formula for compact Riemannian manifolds with smooth boundary. Reilly's identity asserts that if $(M^n, \bar{g})$ is a smooth Riemannian manifold and $\Omega \subset M^n$ is a smooth domain with boundary, then for any $V \in C^\infty(\bar{\Omega})$, the following integral identity holds:

$$
\int_{\Omega}\big[(\bar{\Delta}V)^2-|\bar{\nabla}^2V|^2\big]dv
= \int_{\partial\Omega}\left(\Pi(\nabla z,\nabla z)+2u\Delta z+Hu^2\right)ds
+ \int_{\Omega}\operatorname{Ric}(\nabla V, \nabla V)dv, 
$$
where $\bar{\Delta}$ and $\bar{\nabla}^2$ denote the Laplacian and the Hessian with respect to $\bar{g}$, $z = V|_{\partial \Omega}$, $u = \langle \nabla V, \nu \rangle$, $\nu$ is the outward-pointing unit normal vector and $dv$ and $ds$ stand for the volume elements of $\Omega$ and $\partial \Omega,$ respectively. Here, $\nabla$ and $\Delta$ on $\partial \Omega$ refer to the gradient and the Laplacian of the induced metric on the boundary. Moreover, $\Pi$ and $H$ are the second fundamental form and the mean curvature of $\partial \Omega$, respectively, and $\operatorname{Ric}$ denotes the Ricci curvature of $(M^n, \bar{g})$. This formula has found numerous applications. For instance, Reilly used it to give a new proof of the Alexandrov theorem. Ros~\cite{mulero1987compact} employed Reilly's formula to establish an integral identity, which provided another proof of Alexandrov's rigidity theorem.

\begin{theorem}[\cite{mulero1987compact}]
Let $(M^n,g)$ be a compact $n$-dimensional Riemannian manifold with smooth boundary $\partial M$ and nonnegative Ricci curvature. Let $H$ be the mean curvature of $\partial M$. If $H > 0$ everywhere, then
\begin{equation}\label{eq:ros-inequality}
\int_{\partial M} \frac{1}{H} \, ds \geq n\operatorname{vol}(M),
\end{equation}
with equality in \eqref{eq:ros-inequality} if and only if $M$ is isometric to a Euclidean ball.
\end{theorem}

The proof of this result is based on the Reilly formula and a suitable Dirichlet boundary value problem. Another proof of Alexandrov's theorem was given by Montiel and Ros \cite{montiel1991compact}, who combined Minkowski formulas with the Heintze--Karcher inequality \cite{heintze1978general}. Furthermore, they established an Alexandrov-type theorem for hypersurfaces with constant higher-order mean curvature embedded in the hemisphere $\mathbb{S}^{n}_{+}$ or hyperbolic space $\mathbb{H}^{n}$. However, they could not prove an analogue of inequality \eqref{eq:ros-inequality} in $\mathbb{S}^{n}_{+}$ or $\mathbb{H}^{n}$.

More recently, Brendle \cite{brendle2013constant} obtained a version of Ros's theorem in $\mathbb{S}^{n}_{+}$ and $\mathbb{H}^{n}$, aiming to investigate Alexandrov-type rigidity problems in general relativity and, more broadly, in a wide class of warped product spaces that includes the Schwarzschild manifold. Brendle's approach employs the unit normal flow with respect to a conformal metric to derive a Heintze--Karcher-type inequality. Specifically, Brendle \cite{brendle2013constant} proved the following result.

\begin{theorem}[\cite{brendle2013constant}]
Let $\Omega^n$ be a bounded domain with smooth boundary $\Sigma = \partial \Omega$ in hyperbolic space $\mathbb{H}^{n+1}$. Fix a point $o \in \mathbb{H}^{n+1}$ and set $V = \cosh r(x)$, where $r(x) = d(x, o)$ is the distance to $o$. If the mean curvature $H$ of $\Sigma$ is positive, then
\begin{equation}\label{eq:brendle-inequality}
\int_{\Sigma} \frac{V}{H} \, ds \geq \frac{n+1}{n} \int_{\Omega} V \, dv.
\end{equation}
Equality holds in \eqref{eq:brendle-inequality} if and only if $\Sigma$ is umbilical.
\end{theorem}

Before proceeding, we recall some terminology, in particular the generalized Reilly formula established by Qiu and Xia~\cite{qiu2015generalization}.

\begin{proposition}[\cite{qiu2015generalization}]\label{QuiXiaIden}
Let \((M^n, g)\) be an \(n\)-dimensional compact Riemannian manifold with boundary \(\partial M\). Given functions \(V, u \in C^\infty(M)\) and a constant \(k \in \mathbb{R}\), the following identity holds:
\begin{equation*}
\begin{aligned}
&\int_M V\big[(\Delta_gu + knu)^2 - |\nabla_g^2 u + ku g|^2\big] \, dv \\
&= (n-1)k \int_M (\Delta_gV + nkV)u^2 \, dv \\
&\quad + \int_M \Big[\nabla_g^2 V - (\Delta_gV)g - 2(n-1)k V g + V\operatorname{Ric}_g\Big](\nabla u, \nabla u) \, dv \\
&\quad + \int_{\partial M} V\Big[ 2\frac{\partial u}{\partial\nu}\,\Delta_{\partial M}u + H\Big(\frac{\partial u}{\partial\nu}\Big)^2
        + \Pi(\nabla_{\partial M}u, \nabla_{\partial M}u) + 2(n-1)k\frac{\partial u}{\partial\nu}\,u\Big] ds\\
&\quad + \int_{\partial M} \frac{\partial V}{\partial\nu} \big(|\nabla_{\partial M}u|^2 - (n-1)ku^2\big) \, ds,
\end{aligned}
\end{equation*}
where \(\Pi\) and \(H = \operatorname{tr}(\Pi)\) denote the second fundamental form and the mean curvature of \(\partial M\), respectively.
\end{proposition}
The classical Reilly formula is recovered by setting \(V \equiv 1\) and \(k = 0\). Notably, inequality \eqref{eq:brendle-inequality} can also be derived using this generalized Reilly formula (see \cite{qiu2015generalization}).

Our work is motivated by recent developments concerning functional inequalities on manifolds with boundary that rely on the generalized Reilly formula. For instance, Kwong and Miao (see \cite{kwongmiao2017}, Theorem~3) studied compact Riemannian manifolds $(M^n,g)$ equipped with a static metric and proved a functional inequality relating the static potential, the second fundamental form, and the mean curvature of $\partial M$. More recently, Diógenes, Pinheiro and Ribeiro Jr. (\cite{diogenes2025integral}, Theorem 1) investigated the geometry of critical metrics of the volume functional on compact manifolds with boundary. Their approach combines the generalized Reilly formula of Qiu and Xia \cite{qiu2015generalization} with a suitable boundary value problem, leading to integral inequalities for such critical metrics.

To obtain our results, we apply the generalized Reilly formula of Qiu and Xia \cite{qiu2015generalization} together with a suitable Steklov boundary value problem, following the approach in \cite{cruz2023critical}. Our work was inspired by earlier estimates of Miao--Tam--Xie, Kwong--Miao, and Diógenes--Pinheiro--Ribeiro~Jr. (see \cite{miao2011critical}, Corollary~3.1; \cite{kwongmiao2017}, Theorem~3; and \cite{diogenes2025integral}, Theorem~1). The proof of our main results relies on combining the generalized Reilly formula with the Steklov-type boundary value problem \eqref{usol}. It is worth noting that the computations become significantly more involved due to the complexity of the boundary terms inherent in the general Reilly identity.

\begin{theorem}\label{QXRic}
Let $(M^n, g)$, $n \geq 3$, be an $n$-dimensional connected compact oriented manifold with boundary. Let $V$ be a positive function on $M$ satisfying \eqref{eq:main-system} and assume $\operatorname{Ric}_g\geq (n-1)k g,$ where $k$ is a constant. Then the following inequality holds:
\begin{equation}\label{eqnn0}
\begin{aligned}
&-\int_{\partial M} V\Big[ 2\sigma u\big(\Delta_{\partial M}u + (n-1)ku\big) + H\sigma^2 u^2 + \frac{2H}{n-1}|\nabla_{\partial M}u|^2 - Hku^2 \Big]\, ds \\
&\quad - \tau\int_{\partial M} \big(|\nabla_{\partial M}u|^2 - (n-1)ku^2\big)\, ds - \alpha\sigma\int_{\partial M} u^2\, ds \\
&\geq -k\big(R_g - (n-1)nk\big)\int_{M} V u^2\, dv,
\end{aligned}
\end{equation}
where $u$ is a solution of
\begin{equation}\label{usol}
\left\{
\begin{aligned}
\Delta u + nku &= 0, && \text{in } M,\\[4pt]
\frac{\partial u}{\partial \nu} &= \sigma u, && \text{on } \partial M,
\end{aligned}
\right.
\end{equation}
here $\sigma$ is a constant.
Moreover, equality holds in the inequality \eqref{eqnn0} if and only if
$$
\operatorname{Ric}_g(\nabla u,\nabla u) = (n-1)k |\nabla u|^2 \quad \text{and}\quad \nabla_g^2 u + ku g = 0. \quad
$$
In this case,
$\Delta_{\partial M}u + (n-1)k u = -H\sigma u,
$
and  $$\int_{\partial M} u\, ds_g = 0 \quad \text{or} \quad H\sigma = -(n-1)k.
$$
\end{theorem}

In order to make our approach more transparent, we now fix some terminology (see \cite{miao2009volume}). Let $(M^n, g)$ be a connected compact Riemannian manifold with boundary $\partial M$.  
We say that $g$ is a \emph{Miao--Tam critical metric} (or simply a \emph{critical metric}) if there exists a nonnegative smooth function $V$ on $M^n$ with $V^{-1}(0) = \partial M$ satisfying the overdetermined elliptic system
$$
-(\Delta_g V)g + \nabla^2_g V - V\operatorname{Ric}_g = g .
$$

The next result, which is inspired by Theorem~1 in \cite{diogenes2025integral}, carries important geometric implications. Estimates of this kind provide essential tools for classification results and for excluding certain classes of special metrics on a given manifold.

\begin{corollary}\label{cordne}
Let $(M^n, g, V)$, $n \geq 3$, be an $n$-dimensional connected compact oriented critical metric with connected boundary $\partial M$. Assume that $\operatorname{Ric}_g\geq (n-1)k g$ for some $k \in \mathbb{R}$. Then we have:
\begin{equation}\label{den1}
\begin{aligned}
\frac{1}{H}\int_{\partial M} \Big[|\nabla_{\partial M}u|^2 - (n-1)k u^2 + H^2 u\langle \nabla u, \nabla V\rangle \Big] ds 
\geq -k\big(R - (n-1)nk\big)\int_{M} V u^2\, dv,
\end{aligned}
\end{equation}
where $u$ is a solution of \eqref{usol}.
In particular, if $k \leq 0$, then
\begin{equation}\label{eqdne}
\frac{1}{H}\int_{\partial M} \Big[ |\nabla_{\partial M}u|^2 - (n-1)k u^2 + H^2 u\langle \nabla u, \nabla V\rangle \Big] dv \geq 0.
\end{equation}
Moreover, equality holds in the inequality \eqref{eqdne} if and only if $(M^n, g)$ is isometric to a geodesic ball in a simply connected space form, either hyperbolic space $\mathbb{H}^n$ or Euclidean space $\mathbb{R}^n$.
\end{corollary}

An immediate consequence, obtained by setting $\alpha = \tau = 0$ (see Eq.~(2.3) in \cite{cruz2019prescribing}), is the following.

\begin{corollary}\label{CQXRic}
Let $(M^n, g)$, $n \geq 3$, be an $n$-dimensional connected compact oriented manifold with boundary. Let $V$ be a positive function on $M$ satisfying \eqref{eq:main-system} with $\alpha = \tau = 0$ and $\operatorname{Ric}_g \geq (n-1)k g,$ where $k$ is a constant. Then we have:
\begin{equation}\label{eq:cor-special}
\begin{aligned}
-\int_{\partial M} V\Big[ &2\sigma u\big(\Delta_{\partial M}u + (n-1)ku\big) + H\sigma^2 u^2 \\
&+ \frac{2H}{n-1}|\nabla_{\partial M}u|^2 - Hku^2 \Big]\, ds
\geq -k\big(R - (n-1)nk\big)\int_{M} V u^2\, dv,
\end{aligned}
\end{equation}
where $u$ is a solution of \eqref{usol}.  
Moreover, equality holds in the inequality \eqref{eq:cor-special} if and only if
$$
\operatorname{Ric}_g(\nabla u,\nabla u) = (n-1)k |\nabla u|^2 \quad \text{and}\quad \nabla_g^2 u + ku g = 0. \quad
$$
In this case,
$\Delta_{\partial M}u + (n-1)k u = -H\sigma u,
$
and  $$\int_{\partial M} u\, ds_g = 0 \quad \text{or} \quad H\sigma = -(n-1)k.
$$
\end{corollary}

\begin{remark}
Suppose that $V$ satisfies \eqref{eq:main-system} with constant scalar curvature $R_g$, and assume the constants are related by
$$
R_g \tau = \frac{H\alpha n}{n-1}.
$$
Setting $U = \Delta V$, one verifies that $U$ solves the boundary value problem
\begin{equation}\label{eq:U-system}
\left\{
\begin{aligned}
\Delta U + \frac{R_g}{n-1} U &= 0 && \text{in } M, \\[4pt]
\frac{\partial U}{\partial\nu} &= \frac{H}{n-1} U && \text{on } \partial M.
\end{aligned}
\right.
\end{equation}
Furthermore, choosing $\sigma = \frac{H}{n-1}$, a direct computation shows that the function
$$
u = -\frac{R_g}{n-1} V - \frac{\alpha n}{n-1}
$$
is a solution of \eqref{usol} with parameter $k = \dfrac{R_g}{n(n-1)}$.
\end{remark}

Motivated by the preceding discussion, we now investigate whether the Ricci lower bound $\operatorname{Ric}_g \geq (n-1)k g$ in Theorem~\ref{QXRic} can be replaced by a scalar curvature lower bound $R_g \geq n(n-1)k$. Under such a scalar curvature condition, we prove the following result.

\begin{theorem}\label{QXR}
Let $(M^n, g)$, $n \geq 3$, be an $n$-dimensional connected compact oriented manifold with boundary. Let $V$ be a positive function on $M$ satisfying \eqref{eq:main-system} and assume $R_g \geq n(n-1)k$ for a nonpositive constant $k$. Then we have:
\begin{equation}\label{eqn007}
\begin{aligned}
&-\int_{\partial M} V\Big[ 2\sigma u\big(\Delta_{\partial M}u + (n-1)ku\big)
   + \Big(\frac{n-2}{n-1}H\sigma^2 + 2nk\sigma - Hk\Big)u^2
   + \frac{H}{n-1}|\nabla_{\partial M}u|^2 \Big] ds \\
&\quad + \tau\big(\sigma^2 + (n-1)k\big)\int_{\partial M} u^2\, ds
   - 2\sigma\int_{\partial M} \langle\nabla u,\nabla V\rangle u\, ds \\
&\geq \frac{\alpha}{n-1}\int_M |\nabla u|^2\, dv - \alpha nk\int_M u^2\, dv
   + 2n^2k^2\int_M V u^2\, dv - (3n-2)k\int_M V|\nabla u|^2\, dv,
\end{aligned}
\end{equation}
where $u$ is a solution of \eqref{usol}.  
Moreover, equality holds in \eqref{eqn007} if and only if $R_g = n(n-1)k$ and $\nabla^2 u + ku g = 0$.
\end{theorem}

As a consequence, we obtain a result analogous to Theorem~2 in \cite{diogenes2025integral}.

\begin{corollary}\label{cordne2}
Let $(M^n, g, V)$, $n \geq 3$, be an $n$-dimensional connected compact oriented critical metric with connected boundary $\partial M$. Assume $R_g \geq n(n-1)k$ for a nonpositive constant $k$. Then we have:
\begin{equation}\label{eqn008}
\begin{aligned}
H\int_{\partial M} \langle \nabla V, \nabla u\rangle^2\, ds
- \frac{(n-1)k}{H}\int_{\partial M} u^2\, ds
&\geq \frac{1}{n-1}\int_M |\nabla u|^2\, dv - nk\int_M u^2\, dv \\
&\quad + 2n^2k^2\int_M V u^2\, dv - (3n-2)k\int_M V|\nabla u|^2\, dv,
\end{aligned}
\end{equation}
where $u$ is a solution of \eqref{usol}.  
Moreover, equality holds in \eqref{eqn008} if and only if $R_g = n(n-1)k$ and $\nabla_g^2 u + ku g = 0$.
\end{corollary}

Let us now provide some motivation for our next result. Ambrozio [\cite{ambrozio2017static}, Proposition 6] generalized earlier work of Shen \cite{shen1997note} and Boucher--Gibbons--Horowitz \cite{boucher1984uniqueness} concerning integral identities for static potentials. Specifically, Ambrozio established the following fundamental formula for three-dimensional static triples:

\begin{proposition}[\cite{ambrozio2017static}, Proposition 6]
Let $(M^3, g,V)$ be a compact oriented static triple with scalar curvature $R_g=6$. Denote by $\partial_1M,\dots, \partial_rM$ the connected components of $\partial M$ and by $k_i$ the constant value of $|\nabla V|$ on $\partial_iM$. Then
$$
\sum_{i=1}^r k_i|\partial_iM| + \int_M|\mathring{\operatorname{Ric}}|^2V\,dv = 2\pi\sum_{i=1}^r k_i\chi(\partial_iM),
$$
where $|\partial_iM|$ is the area of $\partial_iM$ and $\chi(\partial_iM)$ is its Euler characteristic.
\end{proposition}

More recently, Cruz and Nunes [\cite{cruz2023critical}, Theorem 1.2] proved an analogue of the Shen \cite{shen1997note} and Boucher--Gibbons--Horowitz \cite {boucher1984uniqueness} results for static manifolds with boundary. Our next result extends Ambrozio's formula to the setting of the $V$-static-type equation.

\begin{theorem}\label{generalizedAmb}
Let $(M^3,g)$ be an orientable compact Riemannian manifold with nonempty boundary $\partial M$. 
Suppose there exists a non-identically zero function $V$ on $M$ satisfying \eqref{eq:main-system} for some $\alpha \in \mathbb{R}$ and $\tau =0$. 
Assume that $\Sigma = V^{-1}(0)$ is connected and intersects $\partial M$ nontrivially, i.e., $\Sigma \cap \partial M \neq \emptyset$. 
Let $\Omega$ be a connected component of $M \setminus \Sigma$ on which $V > 0$, and suppose $S=\partial \Omega \setminus \Sigma$. 
Then the following integral identity holds:
\begin{align*}
\int_{\Omega} V |\mathring{\operatorname{Ric}}_g|^2 \, dv 
&= \beta \left[ 2\pi \chi(\Sigma) -  |\Sigma|\left(\frac{R_g}{6}+\dfrac{H^2}{2}+\dfrac{\alpha^2}{4\beta^2}\right) \right]  \\
&\quad  
+ \frac{\alpha}{2} \bigl( L(\Gamma) - H |S| +\dfrac{3}{2}H^2|\Omega|\bigr)+\dfrac{R_gH^2}{4}\displaystyle\int_{\Omega}Vdv-\dfrac{RH}{6}\displaystyle\int_SVds,
\end{align*}
where $\chi(\Sigma)$ is the Euler characteristic of $\Sigma$ and $L(\Gamma)$ denotes the length of $\Gamma = \Sigma \cap S$.
\end{theorem}

\begin{remark}
We observe that if $\alpha=0$, then we recover Theorem 3.1 in \cite{medvedev2025some}.
\end{remark}

The paper is organized as follows. In Section~\ref{preliminares}, we introduce basic notation, recall previous results on the linearization of curvature, and establish several properties of the $V$-static-type equation. Section~\ref{example} provides explicit examples within this framework. Finally, Section~\ref{proofs} contains the proofs of our main results.

\section{Preliminaries}\label{preliminares}

In this section, we recall some basic facts that will be used throughout the text and establish preliminary results. We also present examples of $V$-static-type equations on manifolds with boundary.

Fischer and Marsden \cite{fischer1975deformations} studied deformations of scalar curvature on a smooth manifold $M$. They analyzed the scalar curvature map 
$$
\mathcal{R} \colon \mathcal{M} \longrightarrow C^{\infty}(M),
$$
which assigns to each metric $g \in \mathcal{M}$ its scalar curvature $R_g$. Here $\mathcal{M}$ denotes the space of Riemannian metrics on $M$ and $C^{\infty}(M)$ is the space of smooth functions on $M$.

To understand the local behavior of this map, one considers its linearization at a metric $g$,
$$
L_g \colon S_2(M) \longrightarrow C^{\infty}(M),
$$
given by
$$
L_g(h) = -\Delta_g(\operatorname{tr}_g h) + \operatorname{div}_g\operatorname{div}_g h - \langle h, \operatorname{Ric}_g \rangle,
$$
and its formal $L^2$-adjoint
$$
L_g^* \colon C^{\infty}(M) \longrightarrow S_2(M),
$$
where $S_2(M)$ is the space of symmetric $(0,2)$-tensors on $M$. A direct computation yields
$$
L_g^*(V) = -(\Delta_g V)\, g + \nabla^2_g V - V\operatorname{Ric}_g .
$$

Using the linearization $L_g$ and its adjoint, Fischer and Marsden proved the local surjectivity of the scalar curvature map. We recall that a Riemannian manifold $(M^n,g)$ is called \emph{static} if there exists a nontrivial function $V$ on $M$ satisfying $L_g^*(V) = 0$ (see \cite{corvino2000scalar}). Standard metrics on $\mathbb{R}^n$, $\mathbb{H}^n(k)$, and $\mathbb{S}^n_+(k)$ are all static and admit a positive static potential. It is well known that a static metric necessarily has constant scalar curvature (cf.\ \cite{corvino2000scalar}, Proposition~2.3). Static metrics have been extensively studied in mathematical relativity (see, e.g., \cite{corvino2013deformation, brendle2013constant}).

Motivated by this framework, Cruz and Vit\'orio \cite{cruz2019prescribing} considered a compact Riemannian manifold $(M^n,g)$ with boundary and studied the curvature prescription problem for manifolds with boundary. They showed that the map
$$
g \mapsto (R_g, \; 2H_g)
$$
is surjective for almost all boundary data. Later, Cruz and Silva Santos \cite{cruz2023critical} investigated variational properties of the volume and boundary-area functionals restricted to the space of Riemannian metrics with prescribed curvature. More precisely, given constants $\alpha, \tau \in \mathbb{R}$, they introduced the map
$$
\Psi(g) = \bigl( R_g,\; 2H_g,\; -2\alpha \operatorname{Vol}(g),\; -2\tau \operatorname{Area}(g) \bigr),
$$
defined on the space of Riemannian metrics on a smooth compact manifold with nonempty boundary.

It is well known (cf.\ \cite{cruz2019prescribing}, Section~2) that the first variation of the mean curvature is given by
\begin{equation}\label{eq:varH}
\delta H_g\cdot h = \frac{d}{dt}\Big|_{t=0} H(g(t)) 
= \frac{1}{2}\left((d(\tr_g h) -\div_gh))(\nu) - \operatorname{div}_g\!\big|_{T_{\partial M}} X - \langle \Pi_g, h \rangle\right),
\end{equation}
where $\nu$ is the outward unit normal to $\partial M$, $X$ is the vector field dual to the one-form $\omega(\cdot)=h(\cdot,\nu)$, $\operatorname{tr}_g h = g^{ij}h_{ij}$ is the trace of $h$, and $\Delta_g = \operatorname{tr}_g(\nabla_g^2)$ denotes the Laplacian, with $\nabla_g^2$ being the Hessian operator.

Using Green's formula, one obtains the identity
$$
\langle \delta R_gh, V \rangle_{L^2(M)} - \langle L^*_g V, h \rangle_{L^2(M)}
= \langle B^* V, h \rangle_{L^2(\partial M)} - \langle 2 \delta H_gh, V \rangle_{L^2(\partial M)},
$$
where 
$$
B^*_g(V) = \frac{\partial V}{\partial \nu}g - V\Pi_g \quad \text{on } \partial M.
$$

Recall also that the linearization of the area and of the volume functional are given by
$$
\delta\operatorname{Area}_g\cdot h = \frac{1}{2}\int_{\partial M} \operatorname{tr}_g(h|_{\partial M})\, ds
\quad \text{and} \quad
\delta\operatorname{Vol}_g\cdot h = \frac{1}{2}\int_{M} \operatorname{tr}_g h\, dv .
$$

The linearization of $\Psi(g)$ is denoted by $S_g(h)$, was studied in detail by Cruz and Silva Santos \cite{cruz2023critical} and its formal $L^2$-adjoint $S_g^*$ is given by
$$
S_g^*(V,\alpha,\tau) = \bigl( L^*_g V - \alpha g,\; B^*_g V - \tau g \bigr).
$$

Thus, a triple $(V,\alpha,\tau)$ in the kernel of $S_g^*$ satisfies the PDE system
\begin{equation}\label{eq:main-system}
\left\{
\begin{aligned}
-(\Delta_g V)g + \nabla^2_g V - V\operatorname{Ric}_g &= \alpha g, && \text{in } M, \\[4pt]
\frac{\partial V}{\partial \nu} g - V \Pi_g &= \tau g, && \text{on } \partial M,
\end{aligned}
\right.
\end{equation}
which is a second-order overdetermined elliptic system with boundary condition.

\begin{remark}\label{rmk:volume-mass}
V-static metrics have recently been interpreted in the context of mathematical relativity. In \cite{McCormick2025}, McCormick showed that for asymptotically hyperbolic manifolds, critical points of the volume-renormalised mass over the space of constant scalar curvature metrics are exactly V-static metrics. This extends the classical correspondence between static metrics and critical points of the ADM mass. For manifolds with boundary, the appropriate boundary condition is to fix the Bartnik data $(\partial M, g_{\partial M}, H)$, and the corresponding critical points are precisely V-static metrics satisfying \eqref{eq:main-system} with $\tau = 0$ and $\alpha$ related to the asymptotic value of the potential \cite{McCormick2025}. This variational characterization motivates the rigidity results obtained in Section 4, where we show that under suitable curvature bounds, such V-static metrics must be isometric to geodesic balls in space forms.
\end{remark}

Taking the trace of \eqref{eq:main-system} yields
\begin{equation}\label{eq:trace-system}
\left\{
\begin{aligned}
\Delta_g V &= -\frac{R_g}{n-1} V - \frac{\alpha n}{n-1}, && \text{in } M, \\[4pt]
\frac{\partial V}{\partial \nu} &= \frac{H}{n-1} V + \tau, && \text{on } \partial M.
\end{aligned}
\right.
\end{equation}

Moreover, a straightforward computation gives
\begin{equation}\label{eq:traceless-system}
\left\{
\begin{aligned}
\mathring{\nabla_g^2 V} &= V \mathring{\operatorname{Ric}}_g, && \text{in } M, \\
V \mathring{\Pi_g} &= 0, && \text{on } \partial M,
\end{aligned}
\right.
\end{equation}
where $\mathring{T}$ denotes the traceless part of a tensor $T$.

We say that a Riemannian manifold $(M^n,g)$ with nonempty boundary $\partial M$ satisfies the \emph{$V$-static-type equation} if there exists a nontrivial smooth solution $V$ to the system \eqref{eq:main-system}. We now establish several properties for metrics satisfying \eqref{eq:main-system}, which are analogous to those of $V$-static manifolds with boundary (see \cite{miao2009volume, cruz2023critical, cruz2023static}).

\begin{proposition}\label{propertiees}
Let $(M^n,g)$, $n \geq 3$, be a connected compact Riemannian manifold with nonempty boundary $\partial M$. Suppose there is a non-identically zero function $V$ in the interior and on $\partial M$ satisfying \eqref{eq:main-system} for some $\alpha, \tau \in \mathbb{R}$, and assume that $\Sigma = V^{-1}(0)$ is nonempty. Then the following holds:

\begin{enumerate}[label={(\alph*)}]
    \item The scalar curvature $R_g$ is constant.
    
    \item The mean curvature $H_g$ is constant and $\Pi_g = \dfrac{H_g}{n-1} g$ on $\partial M$.
    
    \item At each point of $\partial M$, we have
    $$
    R_{\partial M} - \frac{n-2}{n-1} H^2 = R_g - 2\operatorname{Ric}_g(\nu, \nu),
    $$
    where $R_{\partial M}$ denotes the scalar curvature of $\partial M$ with the induced metric.
    
    \item Consider the functional on the space of Riemannian metrics defined by
    $$
    \mathcal{F}(g) = \int_M R_g V \, dv + 2\int_{\partial M} H V \, da - 2\alpha \operatorname{Vol}(g) - 2\tau \operatorname{Area}(\partial M, g),
    $$
    where $V$ is a given smooth nontrivial function on $M$. Then $g$ is a critical point of $\mathcal{F}$.
    
    \item On $\partial M$, we have $$\operatorname{Ric}_g(\nu, X) = 0,$$ 
    for any vector $X$ tangent to $\partial M$, where $\nu$ denotes the outward unit normal vector field to $\partial M$.
    
    \item At regular points of $V$, the norm $|\nabla_M V|$ is a positive constant on each connected component of $\Sigma$.
    
    \item $\Sigma$ is totally umbilical with mean curvature $H_\Sigma = \dfrac{\alpha}{|\nabla_M V|}$.
    
    \item $\Sigma$ is an embedded totally umbilical hypersurface in $M$. More precisely:
    \begin{itemize}
        \item If $\Sigma \cap \partial M = \emptyset$, then $\Sigma$ is a totally umbilical hypersurface contained in $\operatorname{int}(M)$.
        \item If $\Sigma \cap \partial M \neq \emptyset$, then each connected component $\Sigma_0$ of $\Sigma$ intersecting $\partial M$ is a capillary totally umbilical hypersurface of $M$.
\end{itemize}
\end{enumerate}
\end{proposition}

\begin{proof}
Items (a), (b), (c) and (d) follow from Proposition 2.2 in \cite{cruz2023critical}. 

Item (e) is a consequence of (b). Indeed, the contracted Codazzi equation states
$$
\operatorname{div}_{\partial M} \Pi_g - dH_g = \operatorname{Ric}_g(\nu, \cdot).
$$
Since $\Pi_g = \frac{H}{n-1} g$ and $H_g$ is constant, we obtain
$$
\operatorname{Ric}_g(\nu, X) = -\frac{n-2}{n-1} X(H_g) = 0
$$
for all $X \in \mathfrak{X}(\partial M)$.

We now prove item (f). At regular points of $V$, the vector field
$$
\nu = -\frac{\nabla_M V}{|\nabla_M V|}
$$
is normal to $\Sigma$. From $\mathring{\nabla^2 V} = V \mathring{\operatorname{Ric}}_g$ we infer that on $\Sigma$,
$$
\nabla^2 V = \frac{\Delta V}{n} g .
$$
Consequently,
$$
X(|\nabla_M V|^2) = 2\langle \nabla_X \nabla_M V, \nabla_M V \rangle 
= 2 \nabla^2 V(X, \nabla_MV) 
=  \frac{2\Delta V}{n} g(X, \nabla_M V) = 0
$$
for any $X \in \mathfrak{X}(\Sigma)$. Hence $|\nabla_MV|$ is constant on each connected component of $\Sigma$.

To prove item (g), we consider an orthonormal frame $\{e_i\}_{i=1}^n$ with $e_n = -\frac{\nabla V}{|\nabla V|} = \nu$. 
Let $\Pi_{ab}$, $1 \leq a, b \leq n-1$, denote the components of the second fundamental form of $\Sigma$. 
By definition and using $\nabla^2 V = \frac{\Delta V}{n} g$ on $\Sigma$, we obtain
\begin{equation}\label{eq:A-second-fundamental}
\Pi_{ab} = \langle \nabla_{e_a} \nu, e_b \rangle 
= -\frac{1}{|\nabla V|} \langle \nabla_{e_a} \nabla V, e_b \rangle 
= -\frac{1}{|\nabla V|} \nabla_a \nabla_b V 
= -\frac{\Delta V}{n|\nabla V|} g_{ab} 
= \frac{\alpha}{(n-1)|\nabla V|} g_{ab}.
\end{equation}
Thus, $\Sigma$ is totally umbilical with mean curvature
$$
H_\Sigma = \frac{\alpha}{|\nabla V|}.
$$

Finally, we turn to item (h). Let $\Sigma_0$ be a connected component of $\Sigma$ such that $\Sigma_0 \cap \partial M \neq \emptyset$. 
To verify that $\Sigma_0$ meets $\partial M$ with a constant angle along $\partial \Sigma_0$, observe that at any point $q \in \partial \Sigma_0 \cap \partial M$,
\begin{align*}
\langle \nabla_M V(q), \nu(q) \rangle 
&= \left\langle \nabla_{\partial M} V(q) + \frac{\partial V}{\partial \nu} \nu(q), \nu(q) \right\rangle \\
&= \frac{\partial V}{\partial \nu}(q) 
= -\frac{H_g}{n-1} V(q) + \tau 
= \tau,
\end{align*}
since $V(q) = 0$ on $\partial \Sigma_0$. This shows that the angle between $\nabla_M V$ and $\nu$ is constant along the intersection.
\end{proof}

Next we study properties of the kernel of $S_g^*$.

\begin{proposition}\label{prop:constant-in-kernel}
Let $(V,\alpha,\tau)$ in $\ker(S_g^*)$ such that $V$ is a nonzero constant function. Then $(M^n,g)$ is Einstein with scalar curvature $R_g = -\alpha n/V$, and the boundary $\partial M$ is totally umbilical with mean curvature $H = -\tau(n-1)/V$.
\end{proposition}

\begin{proof}
Let $V = c$ be a nonzero constant function such that $(V,\alpha,\tau)$ belongs to $\ker(S_g^*)$. From \eqref{eq:main-system} we obtain $\mathring{\operatorname{Ric}}_g = 0$ in $M$ and $\mathring{\Pi}_g = 0$ on $\partial M$. Moreover, using \eqref{eq:traceless-system} we conclude that $R_g =- \alpha n/c$ and $H = -\tau(n-1)/c$.
\end{proof}

A simple converse is the following:

\begin{proposition}\label{prop:Einstein-umbilical-gives-kernel}
If $(M^n, g)$ is Einstein with totally umbilical boundary, then $\ker(S_g^*) \neq \{0\}$.
\end{proposition}

\begin{proof}
Since $\mathring{\operatorname{Ric}}_g = 0$ in $M$ and $\mathring{\Pi}_g = 0$ on $\partial M$, any nonzero constant function $V$ on $M$ satisfies \eqref{eq:traceless-system} with $\alpha=\tau=0$. Hence $(V,0,0) \in \ker(S_g^*)$.
\end{proof}

Motivated by results in \cite{sheng2025static} for the vacuum static equation, we now recall a basic uniqueness lemma.

\begin{lemma}[\cite{sheng2025static}, Lemma 2.5]\label{lem:unique-solution}
There is no nontrivial solution to the boundary value problem
\begin{equation}\label{eq:trivial-solution}
\left\{
\begin{aligned}
L_g^* V &= 0, && \text{in } M, \\[4pt]
V &= \dfrac{\partial V}{\partial \nu}=0, && \text{on } \partial M.
\end{aligned}
\right.
\end{equation}
\end{lemma}

An immediate consequence is the following uniqueness result.

\begin{corollary}\label{cor:boundary-determines}
An element of $\ker(S_g^*)$ is uniquely determined by its boundary data on $\partial M$.
\end{corollary}

\begin{proof}
Let $V_1, V_2 \in \ker(S_g^*)$ with $V_1 = V_2$ on $\partial M$, and set $V = V_1 - V_2$. 
Then $V$ satisfies
$$
\begin{cases} 
L_g^* V = L_g^* V_1 - L_g^* V_2 = \alpha g - \alpha g = 0 & \text{in } M, \\[4pt]
\dfrac{\partial V}{\partial \nu} = \dfrac{\partial V_1}{\partial \nu} - \dfrac{\partial V_2}{\partial \nu} 
= \dfrac{H}{n-1} V_1 + \tau - \dfrac{H}{n-1} V_2 - \tau 
= \dfrac{H}{n-1} (V_1 - V_2) = 0 & \text{on } \partial M,
\end{cases}
$$
where we used the boundary condition from \eqref{eq:trace-system} together with the hypothesis $V_1 = V_2$ on $\partial M$.
Thus $V =\partial V/\partial \nu = 0$ on $\partial M$. By Lemma~\ref{lem:unique-solution}, $V \equiv 0$ in $M$; i.e., $V_1 = V_2$.
\end{proof}





\section{Examples of $V$-static-type equation}\label{example}

Here, we present some examples of the $V$-static-type equation (see \cite{corvino2013deformation, cruz2023critical, sheng2022deformations}).  

\begin{example}
\begin{enumerate}
\item To begin, let $(\Omega,\Sigma)$ be a domain in $(M^n,g)$, $n\geq 3$, where $\Sigma$ is the boundary of $\Omega$. We say that generic conditions are satisfied on $(\Omega,\Sigma)$ when the following system of equations  

\begin{equation}\label{eq:main-systemb}
\left\{
\begin{array}{rccl}
-(\Delta_g V)g + \nabla^2_g V - V\operatorname{Ric}_g & = & \alpha g, & \text{in } \Omega,\\[0.3cm]
\displaystyle\dfrac{\partial V}{\partial \nu} \, g - V \Pi_g & = & \tau g, & \text{on } \Sigma,
\end{array}
\right.
\end{equation}
has only the trivial solution.

We now describe non-generic domains in the Schwarzschild manifold, in the case $\alpha =0$ (see \cite{sheng2022deformations}) with  

\begin{equation}\label{eq:tau-schwarzschild}
\tau=\dfrac{-\dfrac{m^2}{4r^{n-2}}+(n-1)m-r^{n-2}}{r^{n-1}\left(1+\dfrac{m}{2r^{n-2}}\right)^{2+\frac{2}{n-2}}}.
\end{equation}
and 
\begin{equation}\label{VS}
V=\dfrac{1-\frac{m}{2r^{n-2}}}{1+\frac{m}{2r^{n-2}}}.
\end{equation}

Consider the Schwarzschild metric  

$$
g^S=\left(1+\dfrac{m}{2r^{n-2}}\right)^{\frac{4}{n-2}}\delta
$$
on $\mathbb{R}^n\setminus\{0\}$, where $r=|x|$. Let $\Sigma$ be the Euclidean sphere centered at $0$ with radius $r$ satisfying \eqref{eq:tau-schwarzschild}. The $n$-dimensional Schwarzschild manifold is conformal to $\mathbb{R}^n$.
We claim that the region with boundary $\Sigma$ is a non-generic domain in the Schwarzschild manifold with $\alpha=0$, where $\tau$ and $V$ are given by \eqref{eq:tau-schwarzschild} and \eqref{VS}, respectively. 

In fact, as mentioned in \cite{sheng2022deformations}, the function  defined in  \eqref{VS} $V\in\ker{L_g^*}$, and indeed it is the only possible static potential on connected open domains, as found by Corvino (\cite{corvino2013deformation}). We now verify the boundary condition on Euclidean spheres centered at $0$.  

Let $e$ be a unit vector in the Euclidean norm. Parametrize the Euclidean sphere as $\{x=re\}$. The outward unit normal at $x$ is $\nu=c e$, where  
$$
\nu=\left(1+\dfrac{m}{2r^{n-2}}\right)^{-\frac{2}{n-2}}e.
$$
Since $g^S$ is conformal to the Euclidean metric $\delta$, the mean curvature with respect to this normal is  

\begin{align*}
H&=\left(1+\dfrac{m}{2r^{n-2}}\right)^{-\frac{2}{n-2}}
\Bigl(H_0+(n-1)\Bigl\langle\nabla\Bigl(\dfrac{2}{n-2}
\ln\Bigl(1+\dfrac{m}{2r^{n-2}}\Bigr)\Bigr),e\Bigr\rangle_{\delta}\Bigr)\\
&=\left(1+\dfrac{m}{2r^{n-2}}\right)^{-\frac{2}{n-2}}
\Bigl(\dfrac{n-1}{r}-(n-1)\Bigl\langle\dfrac{mx}{r^n\Bigl(1+\dfrac{m}{2r^{n-2}}\Bigr)},e\Bigr\rangle_{\delta}\Bigr)\\
&=(n-1)\left(1+\dfrac{m}{2r^{n-2}}\right)^{-\frac{2}{n-2}}
\left(\dfrac{1-\dfrac{m}{2r^{n-2}}}{r\Bigl(1+\dfrac{m}{2r^{n-2}}\Bigr)}\right)\\
&=(n-1)\left(1+\dfrac{m}{2r^{n-2}}\right)^{-\frac{2}{n-2}} r^{-1}V.
\end{align*}
On the other hand,  

\begin{equation}\label{eq:boundary-cond}
V_{\nu}-\dfrac{H}{n-1}V
=\dfrac{-\dfrac{m^2}{4r^{n-2}}+(n-1)m-r^{n-2}}
{r^{n-1}\left(1+\dfrac{m}{2r^{n-2}}\right)^{2+\frac{2}{n-2}}}
=\tau.
\end{equation}

Moreover, note that $\Sigma$ is minimal with respect to $g^S$ if and only if $V=0$, i.e., $r=\left(\frac{m}{2}\right)^{\frac{1}{n-2}}$. In particular, if $\Sigma$ is a Euclidean sphere centered at $0$ with radius  

$$
r_{\pm}=\left(\dfrac{(n-1)\pm\sqrt{n^2-2n}}{2}m\right)^{\frac{1}{n-2}},
$$

then $\tau=0$.

\item (\cite{cruz2023critical}) Let $\Omega$ be a geodesic ball in $\mathbb{S}^n$ centered at $N=(0,\ldots,0,1)$ with geodesic radius $R\in(0,\pi)$. Let $r$ be the geodesic distance to $N$. The function  
$$
V=a\cos(r)-\dfrac{\alpha}{n-1},
$$
with  
$$
\tau=a\dfrac{\cos(2R)}{\sin R}-\dfrac{\alpha}{n-1}\dfrac{\cos R}{\sin R},
\quad a\in\mathbb{R},
$$
satisfies the system \eqref{eq:main-system}.

\item (\cite{cruz2023critical}) Let $\mathbb{B}$ be the unit Euclidean ball in $\mathbb{R}^n$. The function  

$$
V=-\dfrac{\alpha}{2(n-1)}|x|^2+\langle b, x\rangle-\tau,
$$

where $\alpha,\tau\in\mathbb{R}$ and $b\in\mathbb{R}^n$, is a solution for system \eqref{eq:main-system}.

\item (\cite{cruz2023critical}) Every Ricci-flat metric on $M$ with totally geodesic boundary satisfies \eqref{eq:main-system} with $\alpha=\tau=0$. The space of potential functions is generated by the constant function $1$. Such a metric cannot be critical for the volume and area functionals, since scaling the metric preserves Ricci-flatness and total geodesicity of the boundary.

\item Let $(M^n, g, V)$, $n \geq 3$, be an $n$-dimensional connected compact oriented critical metric with connected boundary $\partial M$. Taking $\alpha=1$ and $\tau=-\dfrac{1}{|\nabla V|}$, we have that $V$ satisfies \eqref{eq:main-system}, because $V=0$ on $\partial M$.
\end{enumerate}
\end{example}

\section{Proof of main results}\label{proofs}
In this section, we present the proofs of Theorems \ref{QXRic}, \ref{QXR} and \ref{generalizedAmb} and Corollaries \ref{cordne}, \ref{CQXRic} and \ref{cordne2}.

\begin{lemma}\label{lemasol}  
Let $(M^n, g)$, $n \geq 3$, be an $n$-dimensional connected compact oriented manifold with boundary. Suppose that $V$ satisfies \eqref{eq:main-system}. Then we have:
\begin{equation}\label{eq:key-identity}
\begin{aligned}
&-\Bigg\{ \int_{\partial M} V\Big[ 2\sigma u\big(\Delta_{\partial M}u + (n-1)ku\big) + H\sigma^2 u^2 + \frac{2H}{n-1}|\nabla_{\partial M}u|^2 - Hku^2 \Big]\, ds\\
&\qquad + \tau \int_{\partial M} \big(|\nabla_{\partial M}u|^2 - (n-1)ku^2\big)\, ds \Bigg\} \\
&= \int_{M} V|\nabla_g^2 u + ku\, g|^2\, dv + 2\int_{M} V\big(\operatorname{Ric}_g - (n-1)k\, g\big)(\nabla u,\nabla u)\, dv \\
&\quad - k\big(R_g - (n-1)nk\big)\int_M V u^2\, dv + \alpha\int_M |\nabla u|^2\, dv - \alpha kn\int_M u^2\, dv,
\end{aligned}
\end{equation}
where $u$ is a solution of \eqref{usol}.
\end{lemma}

\begin{proof}
From the Qiu--Xia identity, see Proposition \ref{QuiXiaIden}, together with \eqref{eq:main-system}, we obtain
\begin{equation}\label{eqqx0}
\begin{aligned}
\int_{M} V\big[(\Delta_g u + nku)^2 &- |\nabla_g^2 u + ku\, g|^2\big]\, dv \\
&= 2\int_{M} V\big(\operatorname{Ric}_g - (n-1)k\, g\big)(\nabla u,\nabla u)\, dv \\
&\quad - k\big(R_g - (n-1)nk\big)\int_M V u^2\, dv + \alpha\int_M |\nabla u|^2\, dv - \alpha kn\int_M u^2\, dv \\
&\quad + \int_{\partial M} V\Big[ 2\frac{\partial u}{\partial\nu}\,\Delta_{\partial M}u + H\left(\frac{\partial u}{\partial\nu}\right)^2 + \Pi(\nabla_{\partial M} u, \nabla_{\partial M} u) \\
&\qquad + 2(n-1)k \frac{\partial u}{\partial\nu}\, u \Big] ds \\
&\quad + \int_{\partial M} \frac{\partial V}{\partial\nu} \big(|\nabla_{\partial M}u|^2 - (n-1)ku^2\big)\, ds .
\end{aligned}
\end{equation}
Substituting \eqref{usol}, \eqref{eq:trace-system} and \eqref{eq:traceless-system} into \eqref{eqqx0} and rearranging terms, we obtain
\begin{equation}\label{eq000}
\begin{aligned}
-\int_{M} V|\nabla_g^2 u + ku\, g|^2\, dv
&= 2\int_{M} V\big(\operatorname{Ric}_g - (n-1)k\, g\big)(\nabla u,\nabla u)\, dv \\
&\quad - k\big(R_g - (n-1)nk\big)\int_M V u^2\, dv + \alpha\int_M |\nabla u|^2\, dv - \alpha kn\int_M u^2\, dv \\
&\quad + \int_{\partial M} V\Big[ 2\sigma u\,\Delta_{\partial M}u + H\sigma^2 u^2 + \frac{H}{n-1}|\nabla_{\partial M}u|^2 \\
&\qquad + 2(n-1)k\sigma u^2 \Big] ds \\
&\quad + \int_{\partial M} \left(\frac{H}{n-1}V + \tau\right)\big(|\nabla_{\partial M}u|^2 - (n-1)ku^2\big)\, ds .
\end{aligned}
\end{equation}
Rearranging terms, we obtain the identity \eqref{eq:key-identity}.
\end{proof}

\subsection{Proof of Theorem \ref{QXRic}}
\begin{proof}
By Lemma \ref{lemasol} together with the assumption $\operatorname{Ric}_g \geq (n-1)k g$, we obtain
\begin{equation}\label{eqn002}
\begin{aligned}
&-\Bigg\{ \int_{\partial M} V\Big[ 2\sigma u\big(\Delta_{\partial M}u + (n-1)ku\big) + H\sigma^2 u^2 + \frac{2H}{n-1}|\nabla_{\partial M}u|^2 - Hku^2 \Big]\, ds \\
&\qquad + \tau \int_{\partial M} \big(|\nabla_{\partial M}u|^2 - (n-1)ku^2\big)\, ds \Bigg\} \\
&\geq -k\big(R_g - (n-1)nk\big)\int_{M} V u^2\, dv + \alpha\int_M |\nabla u|^2\, dS_g - \alpha kn\int_M u^2\, dv .
\end{aligned}
\end{equation}
Now, integrating the identity $\operatorname{div}(\alpha u\nabla u) = \alpha u\Delta u + \alpha|\nabla u|^2$ over $M$ yields
\begin{equation}\label{eq:div-identity}
\alpha\int_{\partial M} u\frac{\partial u}{\partial\nu}\, ds = \alpha\int_M u\Delta u\, dv + \alpha\int_M |\nabla u|^2\, dv .
\end{equation}
Substituting \eqref{usol} into \eqref{eq:div-identity}, we obtain
\begin{equation}\label{eq:alpha-identity}
\alpha\sigma\int_{\partial M} u^2\, ds = -\alpha nk\int_M u^2\, dv + \alpha\int_M |\nabla u|^2\, dv .
\end{equation}

Combining \eqref{eqn002} with \eqref{eq:alpha-identity} gives
\begin{equation}\label{eq:main-inequality}
\begin{aligned}
&-\int_{\partial M} V\Big[ 2\sigma u\big(\Delta_{\partial M}u + (n-1)ku\big) + H\sigma^2 u^2 + \frac{2H}{n-1}|\nabla_{\partial M}u|^2 - Hku^2 \Big]\, ds \\
&\quad -\tau\int_{\partial M} \big(|\nabla_{\partial M}u|^2 - (n-1)ku^2\big)\, ds - \alpha\sigma\int_{\partial M} u^2\, ds \\
&\geq -k\big(R_g - (n-1)nk\big)\int_{M} V u^2\, dv .
\end{aligned}
\end{equation}

Hence, equality holds in \eqref{eq:main-inequality} if and only if
\begin{equation}\label{eq:equality-conds}
\operatorname{Ric}_g(\nabla u, \nabla u) = (n-1)k|\nabla u|^2 \quad \text{and} \quad \nabla^2 u + u k g = 0.
\end{equation}

On the other hand, using the decomposition
$$
\Delta u = \Delta_{\partial M} u + H\frac{\partial u}{\partial\nu} + \nabla^2 u(\nu,\nu)
$$
together with \eqref{usol} and \eqref{eq:equality-conds}, we deduce
\begin{equation}\label{eq:boundary-rel}
-H\sigma u = \Delta_{\partial M} u + (n-1)k u .
\end{equation}
Integrating \eqref{eq:boundary-rel} over $\partial M$ gives
$$
\int_{\partial M} u\, ds = 0 \quad \text{or} \quad -H\sigma = (n-1)k .
$$
Moreover, from \eqref{eq:boundary-rel} we obtain
$$
\int_{\partial M} |\nabla_{\partial M} u|^2\, ds = \big((n-1)k + \sigma H\big)\int_{\partial M} u^2\, ds .
$$
In particular, $\sigma H \geq -(n-1)k$. If equality holds, then $u$ is constant on $\partial M$. Since $u$ satisfies \eqref{usol}, this implies either $\sigma = 0$ or $u \equiv 0$ on $\partial M$. 
\end{proof}

\subsection{Proof of Corollary \ref{cordne}} 
\begin{proof}
Since $V = 0$ on $\partial M$ and under the hypothesis of Theorem \ref{QXRic}, we obtain 
\begin{equation}\label{eq:cor-base}
-\tau\int_{\partial M} \big(|\nabla_{\partial M}u|^2 - (n-1)ku^2\big)\, ds - \alpha\sigma\int_{\partial M} u^2\, ds_g \geq -k\big(R_g - (n-1)nk\big)\int_{M} V u^2\, dv .
\end{equation}
Using that $(M^n, g, V)$ is a critical metric, we then have $\alpha = 1$.  
Let $\nu = -\frac{\nabla V}{|\nabla V|}$ be the outward unit normal to $\partial M$ and $H = \frac{1}{|\nabla V|}$ the corresponding mean curvature. A direct computation gives
$$
\frac{\partial V}{\partial\nu} = -|\nabla V| = -\frac{1}{H} = \tau .
$$

Using these facts together with \eqref{usol} and \eqref{eq:cor-base}, we obtain
$$
\sigma\int_{\partial M} u^2\, ds = \int_{\partial M} u\langle\nabla u,\nu\rangle\, ds = -\frac{1}{|\nabla V|}\int_{\partial M} u\langle\nabla u,\nabla V\rangle\, ds = -H\int_{\partial M} u\langle\nabla u,\nabla V\rangle\, ds .
$$
Substituting this into \eqref{eq:cor-base} yields
\begin{equation}\label{den}
\frac{1}{H}\int_{\partial M} \Big[ |\nabla_{\partial M}u|^2 - (n-1)k u^2 + H^2 u\langle\nabla u,\nabla V\rangle \Big] ds \geq -k\big(R_g - n(n-1)k\big)\int_M V u^2\, dv.
\end{equation}
Thus, \eqref{den} is established.

We now consider the equality case. By assumption, $R_g \geq n(n-1)k$ and $k \leq 0$. Therefore \eqref{den} gives us
\begin{equation}\label{eq:cor-final}
\frac{1}{H}\int_{\partial M} \Big[ |\nabla_{\partial M}u|^2 - (n-1)k u^2 + H^2 u\langle\nabla u,\nabla V\rangle \Big] ds \geq 0 .
\end{equation}
Hence, \eqref{den1} is proved.

Observe that equality holds in \eqref{eq:cor-final} if and only if
$$
\operatorname{Ric}_g(\nabla u, \nabla u) = (n-1)k|\nabla u|^2, \quad R_g = n(n-1)k, \quad \text{and} \quad \nabla_g^2 u + kug = 0 .
$$
Since $\operatorname{Ric} \geq (n-1)k g$ and $R_g = (n-1)nk$, we conclude that $(M^n, g)$ is Einstein.  
Finally, applying [\cite{miao2011einstein}, Theorem 1.1] we deduce that $(M^n, g)$ is isometric to a geodesic ball in a simply connected space form, either hyperbolic space $\mathbb{H}^n$ or Euclidean space $\mathbb{R}^n$.
\end{proof}

\subsection{Proof of Corollary \ref{CQXRic}}
\begin{proof}
This is an immediate consequence of Theorem \ref{QXRic} by taking $\alpha=\tau=0$. 
\end{proof}

\subsection{Proof of Theorem \ref{QXR}}
\begin{proof}
We follow the argument of Theorem 2 in \cite{diogenes2025integral}, though the situation is slightly more complicated because \(V = 0\) does not hold on \(\partial M\). First, using Lemma \ref{lemasol} together with \eqref{eq:main-system}, we obtain

\begin{eqnarray}\label{eq002}
&-&\displaystyle\int_{\partial M} V\Big[ 2\sigma u\big(\Delta_{\partial M}u + (n-1)ku\big) + H\sigma^2 u^2 + \frac{2H}{n-1}|\nabla_{\partial M}u|^2 - Hku^2 \Big]\, ds\nonumber \\
&\quad& -\tau \int_{\partial M} \big(|\nabla_{\partial M}u|^2 - (n-1)ku^2\big)\, ds\nonumber \\
&=&\displaystyle\int_M V|\nabla^2 u + ku g|^2\, dv + 2\int_M (-\Delta_g V\, g + \nabla^2 V)(\nabla u, \nabla u)\, dv - \alpha\int_M |\nabla u|^2\, dv\nonumber\\
&\quad& - 2(n-1)k\int_M V|\nabla u|^2\, dv - k\big(R_g - (n-1)nk\big)\int_M V u^2\, dv - \alpha kn\int_M u^2\, dv .
\end{eqnarray}

Note that  
\[
\nabla_g^2 V(\nabla u, \nabla u) = \operatorname{div}\!\big(\langle\nabla V,\nabla u\rangle\nabla u\big) - \Delta_g u\,\langle\nabla V,\nabla u\rangle - \nabla_g^2 u(\nabla u,\nabla V).
\]
Integrating this identity and using \eqref{usol} gives

\begin{equation}\label{eqr000}
\int_M \nabla_g^2 V(\nabla u, \nabla u)\, dv = \sigma\int_{\partial M} \langle\nabla V,\nabla u\rangle u\, ds + nk\int_M  u\,\langle\nabla V,\nabla u\rangle\, dv - \frac12\int_M \langle\nabla|\nabla u|^2,\nabla V\rangle\, dv.
\end{equation}

Since \(\operatorname{div}\!\big(|\nabla u|^2\nabla V\big) = \langle\nabla|\nabla u|^2,\nabla V\rangle + |\nabla u|^2\Delta_g V\), integration yields

\begin{equation}\label{eqr001}
\frac12\int_{\partial M} \langle|\nabla u|^2\nabla V,\nu\rangle\, ds = \frac12\int_M \langle\nabla|\nabla u|^2,\nabla V\rangle\, dv + \frac12\int_M |\nabla u|^2\Delta_g V\, dv .
\end{equation}

Combining \eqref{eqr000} with \eqref{eqr001} we infer

\begin{equation}
\begin{aligned}
\int_M \nabla_g^2 V(\nabla u, \nabla u)\, dv &= \sigma\int_{\partial M} \langle\nabla V,\nabla u\rangle u\, ds_g + nk\int_M \langle\nabla V,\nabla u\rangle u\, dv+ \frac12\int_M |\nabla u|^2\Delta_g V\, dv \\
&\quad - \frac12\int_{\partial M} \langle|\nabla u|^2\nabla V,\nu\rangle\, ds .
\end{aligned}
\end{equation}

Consequently,

\begin{equation}\label{eqr002}
\begin{aligned}
\int_M \big[-\Delta_g V\, g + \nabla_g^2 V\big](\nabla u,\nabla u)\, dv &= \sigma\int_{\partial M} \langle\nabla V,\nabla u\rangle u\, ds + nk\int_M \langle\nabla V,\nabla u\rangle u\, dv \\
&\quad - \frac12\int_M |\nabla u|^2\Delta_g V\, dv - \frac12\int_{\partial M} |\nabla u|^2\frac{\partial V}{\partial\nu}\, ds .
\end{aligned}
\end{equation}

From \eqref{eq002} and \eqref{eqr002} we obtain

\begin{equation}\label{eqr003}
\begin{aligned}
&-\int_{\partial M} V\Big[ 2\sigma u\big(\Delta_{\partial M}u + (n-1)ku\big) + H\sigma^2 u^2 + \frac{2H}{n-1}|\nabla_{\partial M}u|^2 - Hku^2 \Big]\, ds \\
&\quad -\tau \int_{\partial M} \big(|\nabla_{\partial M}u|^2 - (n-1)ku^2\big)\, ds \\
&= \int_M V|\nabla_g^2 u + ku g|^2\, dv + 2\sigma\int_{\partial M} \langle\nabla V,\nabla u\rangle u\, ds + 2nk\int_M \langle\nabla V,\nabla u\rangle u\, dv \\
&\quad - \int_M |\nabla u|^2\Delta_g V\, dv - \int_{\partial M} |\nabla u|^2\frac{\partial V}{\partial\nu}\, ds - \alpha\int_M |\nabla u|^2\, dv \\
&\quad - 2(n-1)k\int_M V|\nabla u|^2\, dv - k\big(R_g - (n-1)nk\big)\int_M V u^2\, dv - \alpha kn\int_M u^2\, dv .
\end{aligned}
\end{equation}

Substituting \eqref{eq:trace-system} into \eqref{eqr003} gives

\begin{equation}\label{eqr005}
\begin{aligned}
&-\int_{\partial M} V\Big[ 2\sigma u\big(\Delta_{\partial M}u + (n-1)ku\big) + H\sigma^2 u^2 + \frac{2H}{n-1}|\nabla_{\partial M}u|^2 - Hku^2 \Big]\, ds \\
&\quad -\tau \int_{\partial M} \big(|\nabla_{\partial M}u|^2 - (n-1)ku^2\big)\, ds \\
&= \int_M V|\nabla_g^2 u + ku g|^2\, dv - k\big(R_g - (n-1)nk\big)\int_M V u^2\, dv + \frac{\alpha}{n-1}\int_M |\nabla u|^2\, dv - \alpha nk\int_M u^2\, dv \\
&\quad + 2nk\int_M u\langle\nabla V,\nabla u\rangle\, dv + \frac{1}{n-1}\big(R_g - 2(n-1)^2k\big)\int_M V|\nabla u|^2\, dv \\
&\quad + 2\sigma\int_{\partial M} \langle\nabla V,\nabla u\rangle u\, ds - \frac{H}{n-1}\int_{\partial M} V|\nabla u|^2\, ds - \tau\int_{\partial M} |\nabla u|^2\, ds.
\end{aligned}
\end{equation}

Observe that  
\[
\operatorname{div}(Vu\nabla u) = \langle\nabla(Vu),\nabla u\rangle + Vu\Delta u = V|\nabla u|^2 + u\langle\nabla V,\nabla u\rangle - nkVu^2,
\]
where we have used \(\Delta u = -nku\). Integrating and using \eqref{usol} yields

\begin{equation}\label{eq4}
\int_M u\langle\nabla V,\nabla u\rangle\, dv = \sigma\int_{\partial M} V u^2\, ds - \int_M V|\nabla u|^2\, dv + nk\int_M V u^2\, dv.
\end{equation}

Since \(\frac{\partial u}{\partial\nu} = \sigma u\), we have  
\(\nabla u = \nabla^{\partial M}u + \frac{\partial u}{\partial\nu}\nu = \nabla^{\partial M}u + \sigma u\nu\). Consequently,

\begin{equation}\label{intgrad}
\int_{\partial M} |\nabla u|^2\, ds = \int_{\partial M} |\nabla_{\partial M}u|^2\, ds + \sigma^2\int_{\partial M} u^2\, ds .
\end{equation}

Substituting \eqref{eq4} and \eqref{intgrad} into \eqref{eqr005} we obtain

\begin{equation}\label{eqr006}
\begin{aligned}
&-\int_{\partial M} V\Big[ 2\sigma u\big(\Delta_{\partial M}u + (n-1)ku\big) + \Big(\frac{n-2}{n-1}H\sigma^2 + 2nk\sigma - Hk\Big)u^2 + \frac{H}{n-1}|\nabla_{\partial M}u|^2 \Big]\, ds \\
&\quad + \tau\big(\sigma^2 + (n-1)k\big)\int_{\partial M} u^2\, ds - 2\sigma\int_{\partial M} \langle\nabla u,\nabla V\rangle u\, ds \\
&= \int_M V|\nabla_g^2 u + ku g|^2\, dv - k\big(R_g - (n-1)nk\big)\int_M V u^2\, dv + \frac{\alpha}{n-1}\int_M |\nabla u|^2\, dv - \alpha nk\int_M u^2\, dv \\
&\quad + 2n^2k^2\int_M V u^2\, dv + \frac{1}{n-1}\big(R_g - 2(n-1)(2n-1)k\big)\int_M V|\nabla u|^2\, dv .
\end{aligned}
\end{equation}

Thus, using the assumptions \(R_g \geq (n-1)nk\) and \(k \leq 0\) in \eqref{eqr006}, we deduce

\begin{equation}\label{eqri007}
\begin{aligned}
&-\int_{\partial M} V\Big[ 2\sigma u\big(\Delta_{\partial M}u + (n-1)ku\big) + \Big(\frac{n-2}{n-1}H\sigma^2 + 2nk\sigma - Hk\Big)u^2 + \frac{H}{n-1}|\nabla_{\partial M}u|^2 \Big]\, ds \\
&\quad + \tau\big(\sigma^2 + (n-1)k\big)\int_{\partial M} u^2\, ds - 2\sigma\int_{\partial M} \langle\nabla u,\nabla V\rangle u\, ds\\
&\geq \frac{\alpha}{n-1}\int_M |\nabla u|^2\, dv- \alpha nk\int_M u^2\, dv + 2n^2k^2\int_M V u^2\, dv - (3n-2)k\int_M V|\nabla u|^2\, dv.
\end{aligned}
\end{equation}

Finally, equality holds in the inequality \eqref{eqri007} if and only if \(R_g = n(n-1)k\) and \(\nabla^2 u + ku g = 0\). This completes the proof of the theorem.

\end{proof}

\subsection{Proof of Corollary \ref{cordne2}}
\begin{proof}
Since $(M^n,g)$ is a critical metric, we have $V = 0$ on $\partial M$ and $\alpha = 1$. Then, by Theorem \ref{QXR}, we obtain

\begin{equation}\label{eqr007}
\begin{aligned}
\tau\bigl(\sigma^2 + (n-1)k\bigr) \int_{\partial M} u^2 \, ds
&- 2\sigma \int_{\partial M} \langle \nabla u, \nabla V \rangle u \, ds \\
&\geq \frac{1}{n-1} \int_M |\nabla u|^2 \, dv - nk \int_M u^2 \, dv \\
&\quad + 2n^2k^2 \int_M V u^2 \, dv - (3n-2)k \int_M V|\nabla u|^2 \, dv.
\end{aligned}
\end{equation}

On the other hand, taking into account that $ \nu=-\frac{\nabla V}{|\nabla V|}$ is the outward unit normal to $\partial M$, $H = \frac{1}{|\nabla V|}$ is the mean curvature of $\partial M$ with respect to $\nu$, and $\dfrac{\partial u}{\partial \nu}=\sigma u$, one sees that 
$$\dfrac{\partial V}{\partial \nu}=-\dfrac{1}{H}:=\tau.$$
Moreover,
\begin{equation}\label{eq009}
\sigma u = \langle \nabla u, \nu \rangle = -\frac{1}{|\nabla V|}\langle \nabla u, \nabla V \rangle = -H \langle \nabla u, \nabla V \rangle,
\end{equation}
and consequently

\begin{equation}\label{eq010}
\tau\sigma^2 u^2 = -H \langle \nabla u, \nabla V \rangle^2.
\end{equation}

Substituting \eqref{eq009} and \eqref{eq010} into \eqref{eqr007} yields \eqref{eqn008}, which completes the proof.

\end{proof}

\subsection{Proof of Theorem \ref{generalizedAmb}}

\begin{proof}
To prove this result, we will use the following Pohozaev-type identity due to Schoen \cite{schoen1988existence}:

\begin{theorem} [\cite{schoen1988existence}]
\label{SchoenPoho} Let $(M^n, g)$ be a compact Riemannian manifold with boundary. If $X$ is a smooth vector field on $M$, then
$$
\frac{n-2}{2n} \int_M X(R_g) \, dv 
= -\frac{1}{2} \int_M \langle \mathcal{L}_{X} g, \mathring{\mathrm{Ric}}_g \rangle \, dv 
+ \int_{\partial M} \mathring{\mathrm{Ric}}_g(X, \nu) \, ds,
$$
where $\mathring{\mathrm{Ric}}$ denotes the trace‑free part of the Ricci tensor of $M$.
\end{theorem}

Let $\Omega$ be a connected component of $M\setminus \Sigma$ and let $S$ denote the set $\partial\Omega\setminus\Sigma$. We will assume that $V>0$ on $\Omega$. We also denote by $\nu$ and $\xi$ the outward unit normal vector fields to $S$ and $\Sigma$, respectively. 

Since $R_g$ is constant, taking $X=\nabla^M V$ in Theorem \ref{SchoenPoho} yields
$$
\int_{\Omega} \langle \nabla^2 V, \mathring{\operatorname{Ric}}_g \rangle \, dv 
= \int_{\partial \Omega} \mathring{\operatorname{Ric}}_g(\nabla^M V, \nu) \, ds,
$$
where we have used that $\mathcal{L}_{\nabla^M V} g = 2 \nabla^2 V$. Indeed,
$$
(\mathcal{L}_{\nabla^M V} g)(Y, Z) 
= \langle \nabla_Y \nabla^M V, Z \rangle + \langle \nabla_Z \nabla^M V, Y \rangle 
= 2 \nabla^2 V(Y, Z),
$$
for all $Y,Z\in \mathfrak{X}(M)$. A direct computation shows that $\mathring{\nabla^2 V} = V \mathring{\operatorname{Ric}}_g$, which together with the decomposition
$$
\nabla^2 V = V \mathring{\operatorname{Ric}}_g + \frac{\Delta V}{3} g
$$
implies $\langle \nabla^2 V, \mathring{\operatorname{Ric}}_g \rangle = V |\mathring{\operatorname{Ric}}_g|^2$. Consequently,
\begin{equation}\label{eq:integral-identity}
\int_{\Omega} V |\mathring{\operatorname{Ric}}_g|^2 \, dv 
= \int_{S} \mathring{\operatorname{Ric}}_g(\nabla^M V, \nu) \, ds 
+ \int_{\Sigma} \mathring{\operatorname{Ric}}_g(\nabla^M V, \xi) \, ds,
\end{equation}
where $\xi = -\nabla^M V/|\nabla^M V|$ is the unit normal on $\Sigma=V^{-1}(0)$.

We observe that $\nabla^M V=\nabla^{\partial M}V+\dfrac{\partial V}{\partial\nu}\nu$ and by Proposition \ref{propertiees} item e), we infer $\operatorname{Ric}(X,\nu)=0$ for all $X\in\mathfrak{X}(\partial M)$. Thus, on $S$, we have 
\begin{eqnarray}\label{ricnunu}
\mathring{\operatorname{Ric}}_g(\nabla^M V, \nu) 
&=& \frac{\partial V}{\partial \nu} \mathring{\operatorname{Ric}}_g(\nu, \nu)\nonumber\\
&=& \frac{\partial V}{\partial \nu} \left( \operatorname{Ric}_g(\nu, \nu) - \frac{R_g}{3} \right)\nonumber\\
&=&  \frac{H}{2} V \left( \operatorname{Ric}_g(\nu, \nu) - \frac{R_g}{3} \right),
\end{eqnarray}
where we used the boundary condition from \eqref{eq:trace-system}.

Using \eqref{eq:main-system} and the identity 
$$
\Delta_g V = \Delta_S V + H\nu(V) + \nabla^2 V(\nu, \nu),
$$ 
we obtain
$$
V \operatorname{Ric}_g(\nu, \nu) = -\Delta_S V - H \nu(V) - \alpha.
$$
Substituting this into \eqref{ricnunu} gives
\begin{eqnarray}\label{rictracenu}
\mathring{\operatorname{Ric}}_g(\nabla^M V, \nu) 
= -\frac{H}{2} \left( \Delta_S V + H\nu(V) + \alpha + \frac{R_g V}{3} \right).
\end{eqnarray}

Note that on $\Sigma$, we have $\nabla^M V = -|\nabla^M V| \xi = -\beta \xi$, where $\beta = |\nabla^M V|$ is a positive constant. Thus,
\begin{eqnarray}\label{Ricnablaxi}
\mathring{\operatorname{Ric}}_g(\nabla^M V, \xi) 
= -\beta \mathring{\operatorname{Ric}}_g(\xi, \xi) = -\beta \left( \operatorname{Ric}_g(\xi, \xi) - \frac{R_g}{3} \right).
\end{eqnarray}

By Proposition \ref{propertiees}, we know that $\Sigma$ is totally umbilical, so $A^\Sigma = \dfrac{H_\Sigma}{2} g_\Sigma$, and the Gauss equation gives
$$
R_g - 2 \operatorname{Ric}_g(\xi, \xi) = 2K_\Sigma - \frac{1}{2} H_\Sigma^2,
$$
where $K_\Sigma$ is the sectional curvature of $\Sigma$ and $H_\Sigma = \dfrac{\alpha}{\beta}$. Consequently,
$$
\operatorname{Ric}_g(\xi, \xi) = -K_\Sigma + \frac{R_g}{2} + \frac{H_\Sigma^2}{4} 
= \frac{R_g}{2} - K_\Sigma + \frac{\alpha^2}{4\beta^2}.
$$
This, together with \eqref{Ricnablaxi}, yields
\begin{eqnarray}\label{eqricnaxi}
\int_{\Sigma} \mathring{\operatorname{Ric}}_g(\nabla^M V, \xi) \, ds
= \beta\int_{\Sigma}K_{\Sigma}\,ds- \dfrac{1}{6}R_g\beta |\Sigma| - \dfrac{\alpha^2}{4\beta}|\Sigma|,
\end{eqnarray}
where $|\Sigma|$ denotes the area of $\Sigma.$

Now, let $\Gamma = \Sigma \cap S$ be a parametrized curve. Its geodesic curvature in $\Sigma$ is given by
$$
k_\Gamma = \langle \nabla^\Sigma_{\Gamma'} \Gamma', \nu \rangle 
= A^S(\Gamma', \Gamma') - A^\Sigma(\Gamma', \Gamma') 
= \frac{H}{2} - \frac{H_\Sigma}{2} 
= \frac{H}{2} - \frac{\alpha}{2\beta}.
$$
In particular, we deduce that 
\begin{eqnarray}\label{H2}
\frac{H}{2} = k_\Gamma + \frac{\alpha}{2\beta}.
\end{eqnarray}
Integrating by parts, we obtain
$$
\frac{H}{2} \int_S \Delta_S V  \, ds
= \frac{H}{2} \int_{\Gamma} \frac{\partial V}{\partial \xi} \, ds 
= \frac{H}{2} \int_{\Gamma} \langle \nabla V, \xi \rangle \, ds 
= -\beta \int_{\Gamma} \frac{H}{2} \, ds.
$$
Plugging \eqref{H2} into the above expression yields
\begin{eqnarray}\label{H2delta}
\frac{H}{2} \int_S \Delta_S V  \, dv = -\beta \int_{\Gamma} k_\Gamma \, ds - \frac{\alpha}{2} L(\Gamma),
\end{eqnarray}
where $L(\Gamma)$ denotes the length of $\Gamma$.

Combining \eqref{eq:trace-system}, \eqref{rictracenu}, \eqref{H2delta} and integrating by parts, we obtain
\begin{eqnarray}\label{eqricnu}
\int_{S} \mathring{\operatorname{Ric}}_g(\nabla^M V, \nu) \, dv 
&=& \beta\displaystyle\int_{\Gamma}k_{\Gamma}\,ds+\dfrac{\alpha}{2}L(\Gamma)-\dfrac{H^2}{2}\left[\displaystyle\int_{\Omega}\Delta_MV\,dv-\displaystyle\int_{\Sigma}\dfrac{\partial V}{\partial \xi}ds\right]-\dfrac{\alpha H}{2}|S|-\dfrac{RH}{6}\displaystyle\int_SVds\nonumber\\
&=& \beta\displaystyle\int_{\Gamma}k_{\Gamma}\,ds+\dfrac{\alpha}{2}L(\Gamma)+\dfrac{R_gH^2}{4}\displaystyle\int_{\Omega}Vdv+\dfrac{3\alpha H^2}{4}|\Omega|-\beta|\Sigma|\dfrac{H^2}{2}\nonumber\\
&&-\dfrac{\alpha H}{2}|S|-\dfrac{RH}{6}\displaystyle\int_SVds.\nonumber\\
\end{eqnarray}

Finally, substituting \eqref{eqricnaxi} and \eqref{eqricnu} into \eqref{eq:integral-identity} and simplifying, we arrive at
\begin{align*}
\int_{\Omega} V |\mathring{\operatorname{Ric}}_g|^2 \, dv 
&= \beta \left[ 2\pi \chi(\Sigma) -  |\Sigma|\left(\frac{R_g}{6}+\dfrac{H^2}{2}+\dfrac{\alpha^2}{4\beta^2}\right) \right]  \\
&\quad  
+ \frac{\alpha}{2} \bigl( L(\Gamma) - H |S| +\dfrac{3}{2}H^2|\Omega|\bigr)+\dfrac{R_gH^2}{4}\displaystyle\int_{\Omega}Vdv-\dfrac{RH}{6}\displaystyle\int_SVds,
\end{align*}
where we have used the Gauss-Bonnet theorem, $\chi(\Sigma)$ is the Euler characteristic of $\Sigma$, and $L(\Gamma)$ denotes the length of $\Gamma = \Sigma \cap S$.
\end{proof}

Now we present results concerning the situation where the zero-level set of the potential does not intersect the boundary. These results are analogous to those in \cite{medvedev2024static}, Section 4.

\begin{lemma}\label{lem:component-contains-boundary}
Let $(M^n, g)$, $n \geq 3$, be an $n$-dimensional connected compact oriented manifold with boundary. Let $V$ be a function on $M$ satisfying \eqref{eq:main-system} with $R_g \leq 0$ and $\alpha\leq 0$. 
Assume that the zero set \(\Sigma = V^{-1}(0)\) is contained in the interior of \(M\), i.e., \(\Sigma \subset \operatorname{int}(M)\). 
Then every connected component of \(M \setminus \Sigma\) on which \(V > 0\) contains at least one boundary component of \(M\).
\end{lemma}

\begin{proof}
Consider first the case \(R_g < 0\); the case \(R_g = 0\) is analogous. 
Suppose, for contradiction, that there exists a connected component \(\Omega\) of \(M \setminus \Sigma\) with \(V > 0\) in \(\Omega\) and \(\overline{\Omega} \cap \partial M = \emptyset\). 
From the traced static equation \eqref{eq:trace-system} we have
$$
\Delta_g V = -\frac{R_g}{n-1} V - \frac{\alpha n}{n-1}.
$$
Since \(R_g < 0\) and \(\alpha \leq 0\), it follows that
$$
\Delta_g V \geq -\frac{R_g}{n-1} V > 0 \quad \text{in } \Omega,
$$
so \(V\) is subharmonic in \(\Omega\).

By the weak maximum principle, a nonconstant subharmonic function attains its maximum only on the boundary \(\partial\Omega\). 
But on \(\partial\Omega\) we have \(V = 0\), whereas \(V > 0\) in \(\Omega\). 
Hence \(V\) would achieve a positive interior maximum, contradicting the maximum principle. 
Therefore, such a component \(\Omega\) cannot exist, and consequently every component where \(V > 0\) must intersect \(\partial M\).
\end{proof}

\begin{corollary}\label{cor:boundary-components-bound}
Let $(M^n, g)$, $n \geq 3$, be an $n$-dimensional connected compact oriented manifold with boundary. Let $V$ be a function on $M$ satisfying \eqref{eq:main-system} with $R_g \leq 0$ and $\alpha\leq 0$.
Assume that \(\Sigma = V^{-1}(0) \subset \operatorname{int}(M)\). 
Then the number of connected components of \(M \setminus \Sigma\) on which \(V > 0\) does not exceed the number of connected components of \(\partial M\).
\end{corollary}

\begin{proof}
If the number of positive-level components were greater than the number of boundary components, then by the pigeonhole principle at least one such component would contain no boundary component, contradicting Lemma~\ref{lem:component-contains-boundary}.
\end{proof}

\begin{corollary}\label{cor:cylinder-sigma-connected}
Let $(M^n, g)$, $n \geq 3$, be an $n$-dimensional connected compact oriented manifold with boundary. Let $V$ be a function on $M$ satisfying \eqref{eq:main-system} with $R_g \leq 0$ and $\alpha\leq 0$. Suppose that \(M\) is diffeomorphic to a cylinder, i.e., \(\partial M\) has exactly two connected components. 
If \(\Sigma = V^{-1}(0) \subset \operatorname{int}(M)\), then \(\Sigma\) is connected.
\end{corollary}

\begin{proof}
Assume, to the contrary, that \(\Sigma\) is disconnected. 
Then \(M \setminus \Sigma\) would have at least three connected components. 
Since \(\partial M\) has only two components, Corollary~\ref{cor:boundary-components-bound} would imply that at least one component of \(M \setminus \Sigma\) with \(V > 0\) contains no boundary component, contradicting Lemma~\ref{lem:component-contains-boundary}. Hence \(\Sigma\) must be connected.
\end{proof}

\section*{Data availability statement}
This manuscript has no associated data.

\section*{Conflict of interest statement}
On behalf of all authors, the corresponding author states that there is no conflict of interest.

\section*{Acknowledgments}
M. Andrade was partially supported by the Brazilian National Council for Scientific and Technological Development (CNPq, grants 408834/2023-4, 403869/2024-2, and 400078/2025-2) and FAPITEC/SE/Brazil (grant 019203.01303/2024-1). She also extends her gratitude to the Department of Mathematics at Princeton University, where this work was completed during her visit as a Visiting Fellow, for its warm hospitality. She is especially thankful to Ana Menezes for her continuous encouragement and support.

\end{document}